\documentclass[11pt,intlimits]{amsart}
\usepackage{amssymb}
\usepackage{amsthm}
\usepackage{cite}

\usepackage{fullpage}

\usepackage{esint}

\newtheorem{theorem}{Theorem}[section]
\newtheorem{lemma}[theorem]{Lemma}
\newtheorem{proposition}[theorem]{Proposition}
\newtheorem{conjecture}[theorem]{Conjecture}
\newtheorem{corollary}[theorem]{Corollary}

\newtheorem*{theorem-anon}{Theorem}

\theoremstyle{definition}
\newtheorem*{definition}{Definition}

\theoremstyle{remark}
\newtheorem*{remark}{Remark}

\newcommand{\de}{\partial}
\newcommand{\db}{\overline{\partial}}

\newcommand{\Ric}{\text{Ric}}
\newcommand{\ev}{\text{ev}}
\newcommand{\Bs}{\text{Bs}}

\newcommand{\sections}{H^0(M, K_M^{-m})}

\title{K\"ahler-Einstein metrics and higher alpha-invariants}
\author{Heather Macbeth}
\address{Department of Mathematics, Princeton University; Fine Hall, Washington Rd, Princeton, NJ 08544}
\email{macbeth@math.princeton.edu}

\begin{document}

\maketitle

\begin{abstract}
We give a criterion for the existence of a K\"ahler-Einstein metric on a Fano manifold $M$ in terms of the higher algebraic alpha-invariants $\alpha_{m,k}(M)$.
\end{abstract}

\section{Introduction}

\subsection{Overview}

It has long been understood that on Fano manifolds (that is, compact complex manifolds whose anticanonical bundle is ample), it should be possible to give precise (both necessary and sufficient) and purely algebro-geometric criteria for the existence of a K\"ahler-Einstein metric.  A general such criterion, \emph{K-stability}, was developed conjecturally over many decades, and very recently proved \cite{cds15i,cds15ii,cds15iii,tian12}.

K-stability, however, is in practice very difficult to verify.  For example, it is conjectured but not proven for various moduli spaces of semistable bundles on Riemann surfaces \cite{hwa00, iye11}, and for certain deformations of the Mukai-Umemura manifold \cite[p45]{don08}.  For this reason, it is also natural 
to work on developing simpler and more explicit (though less general) algebro-geometric criteria for the existence of K\"ahler-Einstein metrics on a Fano manifold $M$.

In this paper we develop one such criterion, involving the higher \emph{alpha-invariants} (or \emph{global log-canonical thresholds}) $\alpha_{m,k}(M)$. The model is a theorem of Tian for $k=2$ \cite{tia91},  used by him in proving the existence of K\"ahler-Einstein metrics on the last few dimension-2 manifolds for which that question had been open \cite{tia90}:
\begin{theorem}[\cite{tia90,tia91}, combined with the partial $\mathcal{C}^0$-estimate of \cite{sze13}] \label{tian}
Let $M$ be a Fano manifold of dimension $n\geq 2$.
There exists a natural number $m$, and a (explicitly computable) real number $\epsilon=\epsilon(n,\alpha_{m,2}(M))$, such that if
\begin{eqnarray*}
 \alpha_{m,2}(M)&>&\frac{n}{n+1},\\
 \alpha_{m,1}(M)&>&\frac{n}{n+1}-\epsilon,
\end{eqnarray*}
then $M$ admits a K\"ahler-Einstein metric.
\end{theorem}

\begin{remark}
For clarity we distinguish this from a perhaps better-known theorem previously proved by Tian \cite{tia87}:  Let $M$ be a Fano manifold of dimension $n\geq 2$; if $\alpha(M)>\frac{n}{n+1}$, then $M$ admits a K\"ahler-Einstein metric.
\end{remark}

We give a partial generalization of Tian's work from $k=2$ to arbitrary $k$.
Postponing until Subsection \ref{eigenvalue-conjecture} the definition of our key hypothesis (that for suitable $k$ and $m$, \emph{the $k$-th eigenvalue of $K_M^{-m}$ be controlled}; this is a statement about Bergman metrics),  the criterion is:

\begin{theorem} \label{main-thm}
Let $M$ be a Fano manifold of dimension $n$.
Let $k$ be a natural number, with $2\leq k\leq n$.
Suppose that for each $m$ sufficiently large, the $k$-th eigenvalue of $K_M^{-m}$ is controlled.

Then there exists a natural number $m$, and a (explicitly computable) real number $\epsilon=\epsilon(n,k,\alpha_{m,k}(M))$, such that if
\begin{eqnarray*}
 \alpha_{m,k}(M)&>&\frac{n}{n+1},\\
 \alpha_{m,1}(M)&>&\frac{n}{n+1}-\epsilon,
\end{eqnarray*}
then $M$ admits a K\"ahler-Einstein metric.
\end{theorem}
\noindent The proof of Theorem \ref{main-thm} relies on  Szekelyhidi's deep recent partial $\mathcal{C}^0$-estimate \cite{sze13}, and on a new estimate (Proposition \ref{Einst-alg-approx}) for certain K\"ahler-Einstein metrics.

Tian's work on the case $k=2$ includes, essentially, a proof of control of the second eigenvalue.  We conjecture this hypothesis is likewise valid for all $k$:
\begin{conjecture}\label{lambda-control}
For all $m$ such that $K_M^{-m}$ is very ample, for each natural number $k$ such that $2\leq k\leq n$, the $k$-th eigenvalue of $K_M^{-m}$ is controlled.
\end{conjecture}
This conjecture is simple, natural, and of considerable independent interest, and its proof in general would be expected to combine analytic and algebraic ideas.  We hope to address it in future work.

Still open, and very interesting, is the question of finding a Fano manifold $M$ (or many such $M$) which satisfy the alpha-invariants criterion of Theorem \ref{main-thm}, and which had not previously been known to admit K\"ahler-Einstein metrics.

Since Tian's Theorem \ref{tian} was enough to solve the K\"ahler-Einstein problem for surfaces, such a manifold $M$ would be of complex dimension at least 3.  Recent work in this direction includes computations by Cheltsov, Kosta, Shi and others (e.g.\ \cite{cheltsov09,shi10,ck14}), of  $\alpha_{m,1}(M)$ for some Fano 3-folds $M$ and  $\alpha_{m,2}(M)$ for all Fano 2-folds $M$.  Perhaps some of their methods can be adapted for higher alpha-invariants.

\subsection{Notation}

Throughout this paper $(M, \omega)$ will be a compact K\"ahler manifold of dimension $n$ (sometimes, where specified, with further properties:  projective, Fano, ...).
We write $V$ for $\int_M\omega^n$,  $\fint_M$ for the averaged integration operator $V^{-1}\int_M$, and for a real function $\varphi$, we use the notation
 \begin{eqnarray*}
\omega_{\varphi}&:=&\omega+\sqrt{-1}\de\db\varphi, \\
I_k(\varphi)
&:=& \sum_{r=0}^{k-2} \fint_M \sqrt{-1}\de\varphi\wedge\db\varphi\wedge\omega_{\varphi}{}^r\wedge\omega^{n-r-1}\\
&=&\fint_M \varphi
\left[\omega^{k-1}-\omega_{\varphi}{}^{k-1}\right]\wedge\omega^{n-k+1}.
\end{eqnarray*}

\subsection{Acknowledgements}

I am grateful to Hans-Joachim Hein and my advisor Gang Tian for many discussions on this subject.

\section{Approximation of K\"ahler potentials}

\begin{lemma} \label{expand-diff}
Let $\varphi$ and $\psi$ be K\"ahler potentials. Let $r$ be an integer, $1\leq r\leq n$. Then
\begin{multline*}
\int_M\left(\varphi \ \omega_\varphi{}^r-\psi \ \omega_\psi{}^r\right)\wedge\omega^{n-r}
=\int_M(\varphi-\psi) \ \left[\sum_{i=0}^r\omega_\varphi{}^i\wedge\omega_\psi{}^{r-i}\wedge\omega^{n-r}
-\sum_{i=0}^{r-1} \ \omega_\varphi{}^i\wedge\omega_\psi{}^{r-i-1}\wedge\omega^{n-r+1}\right].
\end{multline*}
\end{lemma}

\begin{proof}
\begin{eqnarray*}
\varphi \ \omega_\varphi{}^r-\psi \ \omega_\psi{}^r
&=&\varphi \ (\omega_\varphi{}^r-\omega_\psi{}^r)+(\varphi-\psi) \ \omega_\psi{}^r\\
&=&\varphi \ \sqrt{-1}\de\db(\varphi-\psi)\left(\sum_{i=0}^{r-1} \ \omega_\varphi{}^i\wedge\omega_\psi{}^{r-i-1}\right)
+(\varphi-\psi) \ \omega_\psi{}^r.
\end{eqnarray*}
Wedge with $\omega^{n-r}$, integrate, and apply Stokes' theorem:
\begin{eqnarray*}
&&\int_M\left(\varphi \ \omega_\varphi{}^r-\psi \ \omega_\psi{}^r\right)\wedge\omega^{n-r}\\
&=&\int_M\left[\varphi \ \sqrt{-1}\de\db(\varphi-\psi)\left(\sum_{i=0}^{r-1} \ \omega_\varphi{}^i\wedge\omega_\psi{}^{r-i-1}\right)
+(\varphi-\psi) \ \omega_\psi{}^r\right]\wedge\omega^{n-r}\\
&=&\int_M\left[(\varphi-\psi)\ \sqrt{-1}\de\db\varphi \left(\sum_{i=0}^{r-1} \ \omega_\varphi{}^i\wedge\omega_\psi{}^{r-i-1}\right)
+(\varphi-\psi) \ \omega_\psi{}^r\right]\wedge\omega^{n-r}\\
&=&\int_M(\varphi-\psi)\left[ (\omega_\varphi-\omega)\left(\sum_{i=0}^{r-1} \ \omega_\varphi{}^i\wedge\omega_\psi{}^{r-i-1}\right)
+ \omega_\psi{}^r\right]\wedge\omega^{n-r}\\
&=&\int_M(\varphi-\psi) \ \left[\sum_{i=0}^r\omega_\varphi{}^i\wedge\omega_\psi{}^{r-i}\wedge\omega^{n-r}
-\sum_{i=0}^{r-1} \ \omega_\varphi{}^i\wedge\omega_\psi{}^{r-i-1}\wedge\omega^{n-r+1}\right].
\end{eqnarray*}
\end{proof}

\begin{corollary} \label{expand-Tsum-diff}
Let $\varphi$ and $\psi$ be K\"ahler potentials. Let $k$ be an integer, $2\leq k\leq n$.  Then
\begin{multline*}
I_{k}(\varphi)-I_{k}(\psi)
=\fint_M(\varphi-\psi) \ \left[\omega^n-\sum_{i=0}^{k-1}\omega_\varphi{}^i\wedge\omega_\psi{}^{k-i-1}\wedge\omega^{n-k+1}
+\sum_{i=0}^{k-2} \ \omega_\varphi{}^i\wedge\omega_\psi{}^{k-i-2}\wedge\omega^{n-r+2}\right].
\end{multline*}
\end{corollary}

\begin{proposition} \label{approximate-T}
Let $c>0$.   Let $k$ be an integer, $2\leq k\leq n$.  There exists $C=C(k, c)$, such that if $\varphi$ and $\psi$ are K\"ahler potentials with
\[
\sup_M\left\lvert\varphi-\psi\right\rvert\leq c,
\]
then
\[
\left\lvert I_{k}(\varphi)-I_{k}(\psi)\right\rvert\leq C.
\]
\end{proposition}

\begin{proof}
By Corollary \ref{expand-Tsum-diff}, for any real number $a$,
\[
I_{k}(\varphi)-I_{k}(\psi)
=\fint_M(\varphi-\psi+a)
\left[\omega^n-\sum_{i=0}^{k-1}\omega_\varphi{}^i\wedge\omega_\psi{}^{k-i-1}\wedge\omega^{n-k+1}
+\sum_{i=0}^{k-2} \ \omega_\varphi{}^i\wedge\omega_\psi{}^{k-i-2}\wedge\omega^{n-r+2}\right].
\]
(The constant $a$ can be added since
\[
\fint_M
\left[\omega^n-\sum_{i=0}^{k-1}\omega_\varphi{}^i\wedge\omega_\psi{}^{k-i-1}\wedge\omega^{n-k+1}
+\sum_{i=0}^{k-2} \ \omega_\varphi{}^i\wedge\omega_\psi{}^{k-i-2}\wedge\omega^{n-r+2}\right]
=1 -(k-1)+(k-2) =0.)
\]

Since $\sup_M|\varphi-\psi|\leq c$,
\[
0\leq\varphi-\psi+c\leq 2c.
\]
Also the forms
\[
\omega^n
+\sum_{i=0}^{k-2} \ \omega_\varphi{}^i\wedge\omega_\psi{}^{k-i-2}\wedge\omega^{n-r+2},
\qquad\qquad
\sum_{i=0}^{k-1}\omega_\varphi{}^i\wedge\omega_\psi{}^{k-i-1}\wedge\omega^{n-k+1}
\]
are positive.  Hence we have the pointwise inequalities
\[
0\leq
(\varphi-\psi+c)\left[\omega^n
+\sum_{i=0}^{k-2} \ \omega_\varphi{}^i\wedge\omega_\psi{}^{k-i-2}\wedge\omega^{n-r+2}\right]
\leq 2c\left[\omega^n
+\sum_{i=0}^{k-2} \ \omega_\varphi{}^i\wedge\omega_\psi{}^{k-i-2}\wedge\omega^{n-r+2}\right],
\]
\[
-2c\left[\sum_{i=0}^{k-1}\omega_\varphi{}^i\wedge\omega_\psi{}^{k-i-1}\wedge\omega^{n-k+1}\right]\leq
-(\varphi-\psi-c)\left[\sum_{i=0}^{k-1}\omega_\varphi{}^i\wedge\omega_\psi{}^{k-i-1}\wedge\omega^{n-k+1}\right]
\leq 0.
\]
Summing, integrating and averaging,
\[
-2(k-1)c\leq I_{k}(\varphi)-I_{k}(\psi)\leq 2(k-1)c.
\]
\end{proof}

\section{Algebraic preliminaries}

In this section $M$ is a variety, $\mathfrak{L}$ an ample line bundle over $M$, and $m$ a natural number.

For each vector subspace $V$ of $H^0(M, \mathfrak{L}^m)$, there is a natural evaluation section $ev_V$ of $V^*\otimes\mathfrak{L}^m$, $
\ev_V(x)=\left(s\mapsto s_x\right)$.
Denote by $\Bs(V)$ the zero locus of this section; that is, the set of points $x\in M$ such that for all sections $s$ in $V$, $s_x=0$.  Projectivizing $\ev_V$ yields a natural map
\[
\iota:=[\ev_V]:M\setminus \Bs(V) \to \mathbb{CP}(V^*).
\]

Let $\omega$ be a K\"ahler metric in $2\pi c_1(\mathfrak{L})$, and $h$ be a hermitian metric on $\mathfrak{L}$ whose curvature is $\omega$.  These induce an inner product $||\cdot ||$ on $V$,
\[
||s||^2
=\int_M |s|^2_{h^m}\omega^n,
\]
hence a Fubini-Study metric on $\mathbb{CP}(V^*)$, which we may pull back under $\iota$ to obtain a nonnegative $(1,1)$-form $\omega_V$ on $M\setminus \Bs(V)$.

\subsection{Definitions}

This inner product, coupled with $h$, also induce a hermitian metric  on $V^*\otimes\mathfrak{L}^m$.  Let
\[
\rho_{\omega,m,V}:M\to\mathbb{R}^{\geq 0}
 \]
be the squared norm of $\ev_V$ with respect to that hermitian metric.  (As the notation suggests, $\rho_{\omega,m,V}$ depends on $\omega$ but not otherwise on $h$.) Equivalently, if $(s_1,\ldots s_k)$ is a basis for $V$ which is orthonormal with respect to $||\cdot||$, then $\rho_{\omega,m,V}
=\sum_{i=1}^k\left\lvert s_i\right\rvert^2_{h^m}$.
Obviously $\rho_{\omega,m,V}$ vanishes precisely on $\Bs(V)$.  It is easily checked that $\tfrac{1}{m} \log\rho_{\omega,m,V}$ is the potential with respect to $\omega$ of the (possibly distributional) pullback $(1,1)$-form $\tfrac{1}{m}\omega_V$:
\[
\tfrac{1}{m}\omega_V=\omega + \tfrac{1}{m}\sqrt{-1}\de\db \log\rho_{\omega,m,V}.
\]

\begin{definition}
The \emph{($m$-th) Bergman kernel} of $\omega$ is $\rho_{\omega,m}:=\rho_{\omega,m,H^0(M, \mathfrak{L}^m)}$.
\end{definition}

\begin{definition}
Let $\mathcal{G}_k$ be the Grassmannian of $k$-dimensional vector subspaces of $H^0(M, \mathfrak{L}^m)$.  The \emph{($(m,k)$-th) alpha-invariant} of $\mathfrak{L}$ is
\[
\alpha_{m,k}(\mathfrak{L})
:=\sup\left\{\alpha>0:\infty>\sup_{V\in \mathcal{G}_k}\int_M\rho_{\omega,m,V}{}^{-\alpha/m}\omega^n\right\}.
\]
In particular, for a Fano manifold $M$, the \emph{($(m,k)$-th) alpha-invariant} of $M$ is $\alpha_{m,k}(M):=\alpha_{m,k}(K_M^{-1})$.
\end{definition}
\noindent It is easily checked that, as implied by the notation, $\alpha_{m,k}(\mathfrak{L})$ is independent of the chosen $\omega$, $h$.

\subsection{Control on Bergman metrics}\label{eigenvalue-conjecture}

In this subsection let  $\mathfrak{L}$ be very ample; let $m$ always be 1, and let $V$ always be the whole vector space $H^0(M, \mathfrak{L})$.  Thus $\Bs(V)=\emptyset$, and
\[
\iota:M\to\mathbb{CP}(V^*)
\]
is a smooth embedding.

Let $\mathcal{M}_\mathfrak{L}\cong GL_{\mathbb{C}}(V)/U(V)$ be the homogeneous space of inner products on $V$.   For an inner product $a\in \mathcal{M}_\mathfrak{L}$, as before there is an induced Fubini-Study metric $\Omega_a$ on $\mathbb{CP}(V)$, and as before there is a \emph{Bergman metric} $\omega_a:=\iota^*\Omega_a$ on $M$ induced by pulling back $\Omega_a$.

Also as before let $\omega$ be a K\"ahler metric in $2\pi c_1(\mathfrak{L})$, and $h$ be a hermitian metric on $\mathfrak{L}$ whose curvature is $\omega$.  These induce a reference inner product $||\cdot ||$ on $V$,
\[
||s||^2
=\int_M |s|^2_{h}\omega^n,
\]

For any $a\in\mathcal{M}_\mathfrak{L}$, we may simultaneously diagonalize $||\cdot||$ and $a$, producing a basis $(s_1,\ldots s_N)$ of V and positive reals $\mu_1(a)\geq\mu_2(a)\geq\cdots>0$  such that
\begin{itemize}
\item $(s_1,\ldots s_N)$  is orthonormal with respect to $||\cdot||$;
\item $(\mu_1(a)^{1/2}s_1,\ldots \mu_N(a)^{1/2}s_N)$  is orthonormal with respect to $a$.
\end{itemize}
It is easily checked that the function $\psi_a:=\log\left(\sum_{j=1}^N\mu_j(a)|s_j|_{h}^2\right)$
is the K\"ahler potential  with respect to $\omega$ of the Bergman metric $\omega_a$:  that is, $\omega_a= \omega+\sqrt{-1}\de\db\psi_a$.

\begin{definition}
Let $2\leq k\leq n$.
\begin{enumerate}
\item Let $D$ be a subset of $\mathcal{M}_\mathfrak{L}$. \emph{The $k$-th eigenvalue of $\mathfrak{L}$ is controlled on $D$}, if for each $\epsilon>0$, there exists
$C=C(n, k, \omega,\epsilon)$, such that for all inner products $a$ in $D$,
\[
\log\left[\frac{\mu_1(a)}{\mu_{k}(a)}\right]
\leq (1+\epsilon)I(\psi_a) + C.
\]
\item \emph{The $k$-th eigenvalue of $\mathfrak{L}$ is controlled}, if it is controlled throughout the full set $\mathcal{M}_\mathfrak{L}$.
\end{enumerate}
\end{definition}

The $k$-th eigenvalue of $\mathfrak{L}$ is obviously controlled on any \emph{compact} subset $D$ of $\mathcal{M}_\mathfrak{L}$; what is not obvious is whether it is controlled on the full, noncompact, $\mathcal{M}_\mathfrak{L}$.

The following non-sharp version of the hypothesis will suffice for a slightly weaker existence theorem:

\begin{definition}
Let $2\leq k\leq n$.
\begin{enumerate}
\item Let $D$ be a subset of $\mathcal{M}_\mathfrak{L}$. \emph{The $k$-th eigenvalue of $\mathfrak{L}$ is weakly controlled on $D$}, if there exists
$C=C(n, k, \omega)$, such that for all inner products $a$ in $D$,
\[
\log\left[\frac{\mu_1(a)}{\mu_{k}(a)}\right]
\leq CI(\psi_a) + C.
\]
\item \emph{The $k$-th eigenvalue of $\mathfrak{L}$ is weakly controlled}, if it is weakly controlled throughout the full set $\mathcal{M}_\mathfrak{L}$.
\end{enumerate}
\end{definition}

\section{Estimates for Einstein potentials by means of algebraic approximation}\label{Einst-alg-approx-sect}

In this section, and throughout the rest of this paper, $M$ is Fano and the K\"ahler metric $\omega$ is in $2\pi c_1(M)$.  Thus $V=\int_M\omega^n=(2\pi)^n c_1(M)^n$.  Let $h$ be a hermitian metric on $K_M^{-1}$ whose curvature is $\omega$.

\subsection{On partial $\mathcal{C}^0$-estimates}

For a class $\mathcal{A}$ of K\"ahler metrics, an estimate of the form
\[
\inf_{\omega\in\mathcal{A}}\rho_{\omega,m}\geq a>0
\]
is called a \emph{partial $\mathcal{C}^0$-estimate}.  Such estimates are expected to hold uniformly for quite general classes of metric.

They are in general proved using convergence theory for classes of manifolds whose metrics satisfy some analytic constraint.  Tian's work on complex surfaces \cite{tia90} included a partial $\mathcal{C}^0$-estimate for K\"ahler-Einstein surfaces, proved using the orbifold convergence of K\"ahler-Einstein 4-manifolds.  Deep, very recent work \cite{cds15i,cds15ii,cds15iii,ds14,sze13,tian12,tia13} has produced partial $\mathcal{C}^0$-estimates for various classes of metrics in arbitrary dimension, proved using Cheeger-Colding theory.

In this paper we will use one of these, Szekelyhidi's partial $\mathcal{C}^0$-estimate  along the Aubin continuity method:

\begin{theorem}[\cite{sze13}]\label{partial-c0}
Let $T\leq1$ be a positive real.  Let $(\omega_t)$, for $t\in(0,T)$, be K\"ahler metrics, such that
\[
\Ric(\omega_{t})=t\omega_{t}+(1-t)\omega.
\]
Then there exist a natural number $m=m(M,\omega)$ and a constant $a=a(M,\omega)$, such that the family $(\omega_t)$ satisfies a partial $\mathcal{C}^0$-estimate:  for all $t\in(0,T)$,
\[\rho_{\omega_{t},m}\geq a.
\]
\end{theorem}

The importance of partial $\mathcal{C}^0$-estimates lies in the following standard result (\cite{tia90}, see also \cite{tos12}), proved using Moser iteration and a Sobolev inequality of Croke and Li.

\begin{proposition}\label{algebraic-approximate-potential}
Let $a$ and $\lambda$ be positive reals.  There exists a constant $C=C(n,m,N,V,a,\lambda)$, such that if
\begin{itemize}
\item $\Ric(\omega)\geq \lambda$,
\item $\varphi\in\mathcal{C}^\infty(M)$ is such that
\begin{itemize}
\item $\omega_\varphi$ is K\"ahler
\item $\Ric(\omega_\varphi)\geq \lambda$
\item $\rho_{\omega_\varphi,m}\geq a$
\end{itemize}
\end{itemize}
then there exist
\begin{itemize}
\item real numbers $1=\lambda_1\geq \lambda_2\geq\cdots\geq \lambda_N>0$
\item an $h$-orthonormal basis $(s_1, s_2, \ldots s_N)$ of $\sections$
\end{itemize}
such that
\[
\sup_M\left\lvert\varphi-\sup_M\varphi
-\tfrac{1}{m}\log\left(\sum_{j=1}^N\lambda_j|s_j|_{h^m}^2\right)\right\rvert\leq C.
\]
\end{proposition}

\subsection{The estimate}

This estimate is a generalization of those in \cite{tia87, tia90, tia91}; a good exposition is available in \cite{tos12}.  Fix a natural number $m$ sufficiently large that $K_M^{-m}$ is very ample.  Let $N$ be the dimension of $\sections$.

\begin{proposition} \label{Einst-alg-approx}
Let $k$ be a natural number, with $k\leq n$.  Fix $\epsilon>0$.  Let $\alpha$, $a$, $\delta$, $A$ be positive reals.  There exists a constant $C=C(n,m,k,\omega,\alpha,a, \delta,\epsilon, A)$, such that if
\begin{enumerate}
\item (partial $\mathcal{C}^0$-estimate) $\varphi$ is a K\"ahler potential, with $\rho_{\omega_\varphi,m}\geq a$,
\item (alpha-invariants criterion) $\alpha_{m,k}(M)>\alpha$,
\item if $k\geq 3$, the $k$-th eigenvalue of $K_M^{-m}$ is controlled, with constant $A$ for ratio $1+\epsilon$;
\end{enumerate}
if, for some real number $t$ with $\delta\leq t \leq 1$,
$\varphi$ solves the Aubin continuity-method equations at $t$:
\[
\begin{cases}
\Ric(\omega_{\varphi})=t\omega_{\varphi}+(1-t)\omega,\\
\displaystyle\int_Me^{t\varphi}\omega_{\varphi}{}^n=V,
\end{cases}
\]
then
\begin{enumerate}
\item \label{k-1} if $k=1$,
\[
\sup_M\varphi
\leq\frac{1-\alpha}{\alpha}\fint_M (-\varphi) \ \omega_{\varphi}{}^n
+ C;
\]
\item \label{k-geq-2} if $k\geq2$,
\[
\sup_M\varphi
\leq\frac{1-\alpha}{\alpha}\fint_M (-\varphi) \ \omega_{\varphi}{}^n
+(1+\epsilon)I_k(\varphi)+ C.
\]
\end{enumerate}
\end{proposition}

\begin{remark}
If the $k$-th eigenvalue of $K_M^{-m}$ is instead only weakly controlled, then the same arguments establish the weaker conclusion
\[
\sup_M\varphi
\leq\frac{1-\alpha}{\alpha}\fint_M (-\varphi) \ \omega_{\varphi}{}^n
+CI_k(\varphi)+ C.
\]
\end{remark}

\begin{proof}
By Jensen's inequality,
\begin{equation}\label{jensen}
\alpha t\sup_M\varphi
+(1-\alpha)\fint_M t\varphi \ \omega_{\varphi}{}^n
\leq\log\left[\fint_M
e^{\alpha t\sup_M\varphi+(1-\alpha)t\varphi}\omega_\varphi{}^n
\right].
\end{equation}

Let $f$ be the real function on $M$ such that
\[
\begin{cases}
\sqrt{-1}\de\db f=\Ric(\omega)-\omega,\\
\displaystyle\int_Me^f\omega^n=V.
\end{cases}
\]
By the Aubin equation on $\varphi$, and by construction of $f$,
\[
e^{t\varphi}\omega_\varphi{}^n=e^f\omega^n;
\]
hence
\begin{eqnarray}
\int_M
e^{\alpha t\sup_M\varphi+(1-\alpha)t\varphi}\omega_\varphi{}^n
&=&\int_M
e^{\alpha t(\sup_M\varphi-\varphi)+f}\omega^n \notag\\
&\leq&C\int_M
e^{\alpha t(\sup_M\varphi-\varphi)}\omega^n.\label{einstein-eqn}
\end{eqnarray}
Using the  partial $\mathcal{C}^0$-estimate and Proposition \ref{algebraic-approximate-potential},
there exist
\begin{itemize}
\item real numbers $1=\lambda_1\geq \lambda_2\geq\cdots\geq \lambda_N>0$
\item an $h$-orthonormal basis $(s_1, s_2, \ldots s_N)$ of $\sections$
\end{itemize}
such that
\begin{equation}\label{approx}
\sup_M\left\lvert\varphi-\sup_M\varphi
-\tfrac{1}{m}\log\left(\sum_{j=1}^N\lambda_j|s_j|_{h^m}^2\right)\right\rvert\leq C.
\end{equation}
Applying the alpha-invariants criterion to (\ref{approx}),
\begin{eqnarray}
\int_Me^{\alpha t(\sup_M\varphi-\varphi)}\omega^n
&\leq&
e^{\alpha t c}\int_M\left(\sum_{j=1}^N\lambda_j|s_j|_{h^m}^2\right)^{-\alpha t /m} \omega^n\notag\\
&\leq&
e^{\alpha t c}\lambda_k^{-\alpha t/m}\int_M\left(\sum_{j=1}^k|s_j|_{h^m}^2\right)^{-\alpha t /m}\omega^n\notag\\
&\leq&
Ce^{\alpha t c}\lambda_k^{-\alpha t/m}.\label{alg-approx}
\end{eqnarray}
Combining equations (\ref{jensen}), (\ref{einstein-eqn}) and (\ref{alg-approx}),
\[
\alpha t\sup_M\varphi
+(1-\alpha)\fint_M t\varphi \ \omega_{\varphi}{}^n
\leq
C+\alpha tC
+ \frac{\alpha t}{m}\log\left(\frac{1}{\lambda_k}\right).
\]
Rearranging, and using that $t\geq \delta$,
\[
\sup_M\varphi
\leq\frac{1-\alpha}{\alpha}\fint_M (-\varphi) \ \omega_{\varphi}{}^n
+\tfrac{1}{m}\log\left(\frac{1}{\lambda_k}\right)+ C.
\]

Now denote by $\psi$ the algebraic K\"ahler potential
\[
\tfrac{1}{m}\log\left(\sum_{j=1}^N\lambda_j|s_j|_{h^m}^2\right).
\]
\begin{enumerate}
\item If $k=1$, then $\log\left(\frac{1}{\lambda_k}\right)=0$ and we are done.
\item If $k\geq 2$, then since the $k$-th eigenvalue of $K_M^{-m}$ is controlled,  with constant $A$ for the ratio $1+\epsilon$,
\[
\tfrac{1}{m}\log\left(\frac{1}{\lambda_{k}}\right)
\leq A+ (1+\epsilon)I_k(\psi),
\]
and by (\ref{approx}) and Proposition \ref{approximate-T}, $I_k(\psi)\leq C+I_k(\varphi)$.  This concludes the proof.
\end{enumerate}
\end{proof}

\section{$\mathcal{C}^0$ estimates from higher alpha invariants}

In this section we maintain the notation of Section \ref{Einst-alg-approx-sect}, and fix a natural number $k$, with $2\leq k\leq n$.

\subsection{Two standard functionals}

We recall the functionals $I$ and $J$ introduced by Aubin:
\begin{eqnarray*}
I(\varphi)&:=& \sum_{r=0}^{n-1} \fint_M \sqrt{-1}\de\varphi\wedge\db\varphi\wedge\omega_{\varphi}{}^r\wedge\omega^{n-r-1};\\
J(\varphi)&:=& \sum_{r=0}^{n-1}\frac{n-r}{n+1} \fint_M \sqrt{-1}\de\varphi\wedge\db\varphi\wedge\omega_{\varphi}{}^r\wedge\omega^{n-r-1}.
\end{eqnarray*}

By construction $I(\varphi)=I_{n+1}(\varphi)$, and for any integer $k$ such that $2\leq k\leq n$,
\begin{eqnarray}
(n+1)J(\varphi)-I(\varphi)
&=& \sum_{r=0}^{n-2}(n-r-1) \fint_M \sqrt{-1}\de\varphi\wedge\db\varphi\wedge\omega_{\varphi}{}^r\wedge\omega^{n-r-1}\notag\\
&\geq&(n-k+1)I_k(\varphi).\label{H-I-J}
\end{eqnarray}

Also by construction, for any K\"ahler potential $\varphi$,
\begin{equation} \label{I-J-diff-pos}
I(\varphi)-J(\varphi)
=\sum_{r=0}^{n-1}\frac{r+1}{n+1} \fint_M \sqrt{-1}\de\varphi\wedge\db\varphi\wedge\omega_{\varphi}{}^r\wedge\omega^{n-r-1}
\geq 0,
\end{equation}
and by Stokes' theorem
\begin{equation}\label{I-as-diff}
I(\varphi)=- \sum_{r=0}^{n-1} \fint_M \varphi\sqrt{-1}\de\db\varphi\wedge\omega_{\varphi}{}^r\wedge\omega^{n-r-1}
=  \fint_M \varphi(\omega^n-\omega_\varphi{}^n).
\end{equation}

\subsection{Aubin continuity method setup}

Throughout the rest of the paper, we study Aubin's system for a K\"ahler potential $\varphi$:
\begin{equation*} \tag{$*_t$}
\begin{cases}
\Ric(\omega_{\varphi})=t\omega_{\varphi}+(1-t)\omega,\\
\displaystyle\int_Me^{t\varphi}\omega_{\varphi}{}^n=V,
\end{cases}
\end{equation*}
\noindent Let $\mathcal{I}$ be the subset of $t\in[0,1]$ for which $(*_t)$ has a solution.

By an implicit-function-theorem argument \cite{aub84}, $\mathcal{I}$ is open in $[0,1]$ and contains $0$; moreover, there exists a  1-parameter family of solutions $(\varphi_t)_{t\in \mathcal{I}\cap [0,1)}$ which is differentiable in $t$.

\begin{remark}
If $M$ has no holomorphic vector fields, then the linearization of the relevant operator is invertible also at $t=1$, and so the implicit-function-theorem argument produces a 1-parameter family of solutions $(\varphi_t)_{t\in \mathcal{I}}$ which is differentiable in $t$; however, we will not need that here.
\end{remark}

\subsection{An opposite estimate}

The following identity for the family $(\varphi_t)$ essentially appears in the proof of Proposition 2.3 in \cite{tia87}; see also the lecture notes \cite{tia97} (Proposition 4.3).

\begin{proposition} For all $t\in \mathcal{I}\cap [0,1)$,
\[
-\frac{1}{t}\int_0^1[I(\varphi_s)-J(\varphi_s)]ds
=J(\varphi_t)-\fint_M\varphi_t\omega^n.
\]
\end{proposition}

\begin{corollary} \label{I-J-ineq} For all $t\in \mathcal{I}\cap [0,1)$,  for any integer $k$ such that $2\leq k\leq n$,
\[
\fint_M (-\varphi_t) \ \omega_{\varphi_t}{}^n
+(n-k+1)I_k(\varphi)
\leq n \ \sup_M\varphi_t.
\]
\end{corollary}

\begin{proof}[Proof of Corollary \ref{I-J-ineq}]
By (\ref{I-J-diff-pos}), $I(\varphi_s)-J(\varphi_s)\geq 0$ for all $s$, so
\begin{eqnarray*}
0&\geq &J(\varphi_t)-\fint_M\varphi_t\omega^n\\
&=&J(\varphi_t)-\frac{n}{n+1}\fint_M\varphi_t\omega^n - \frac{1}{n+1}\left[I(\varphi_t)
+\fint_M\varphi_t\omega_{\varphi_t}{}^n\right].
\end{eqnarray*}
(The second line is an application of (\ref{I-as-diff}).)  Multiplying through by $n+1$ and rearranging,
\[
n\fint_M\varphi_t\omega^n
\geq \left[(n+1)J(\varphi_t)-I(\varphi_t)\right]+\fint_M (-\varphi_t) \ \omega_{\varphi_t}{}^n.
\]
Now apply (\ref{H-I-J}) (on the right-hand side) and the estimate $\sup_M\varphi_t \geq \fint_M\varphi_t\omega^n$ (on the left-hand side).
\end{proof}

\subsection{$\mathcal{C}^0$ estimate}

\begin{theorem} \label{c0-estimate}
Suppose that
\begin{enumerate}
\item (partial $\mathcal{C}^0$-estimate) there exists $a>0$,
such that for all $t\in \mathcal{I}\cap [\delta,1)$, $ \rho_{\omega_{\varphi_t},m}\geq a$;
\item (alpha-invariants criterion)
\begin{eqnarray*}
 \frac{n}{n+1}&<&\alpha_{m,k}(M),\\
 (k-1)\left[\frac{1}{\alpha_{m,1}(M)}-\frac{n+1}{n}\right]
 &<&(n-k+1)\left[\frac{n+1}{n}-\frac{1}{\alpha_{m,k}(M)}\right];
\end{eqnarray*}
\item the $k$-th eigenvalue of $K_M^{-m}$ is controlled.
\end{enumerate}
Then there exists a constant $C$,
such that for all $t\in \mathcal{I}\cap [\delta,1)$,
\[
\sup_M\varphi_t\leq C.
\]
\end{theorem}

\begin{remark}
The second part of the alpha-invariants criterion may be rearranged attractively as
\[
\frac{k-1}{\alpha_{m,1}(M)}+\frac{n-k+1}{\alpha_{m,k}(M)}
 < n+1.
\]
\end{remark}

\begin{proof}  Throughout this proof write $\varphi$ for a potential $\varphi_t$ satisfying the hypotheses of the theorem.
By Corollary \ref{I-J-ineq},
\begin{equation}\label{C}
\fint_M (-\varphi) \ \omega_{\varphi}{}^n
+(n-k+1)I_k(\varphi)
\leq n  \sup_M\varphi.
\end{equation}

Let $\beta_1\geq 0$ and $\beta_k>-1$ be real numbers, to be determined subsequently, with
\[
\beta_1>\frac{1-\alpha_{m,1}(M)}{\alpha_{m,1}(M)},
\qquad\qquad\frac{1}{n}>\beta_k>\frac{1-\alpha_{m,k}(M)}{\alpha_{m,k}(M)}.
\]
Thus
\[
\frac{1-\tfrac{1}{1+\beta_1}}{\tfrac{1}{1+\beta_1}}=\beta_1,
\qquad \frac{1-\tfrac{1}{1+\beta_k}}{\tfrac{1}{1+\beta_k}}=\beta_k,
\qquad 0<\frac{1}{1+\beta_1}<\alpha_{m,1}(M),
\qquad \frac{n}{n+1}<\frac{1}{1+\beta_k}<\alpha_{m,k}(M).
\]
Applying Proposition \ref{Einst-alg-approx} (\ref{k-1})  with $\alpha=1/(1+\beta_1)$,  we obtain:
\begin{equation}\label{A}
\sup_M\varphi
\leq\beta_1\fint_M (-\varphi) \ \omega_{\varphi}{}^n
+ C;
\end{equation}
applying Proposition \ref{Einst-alg-approx} (\ref{k-geq-2}) with $\alpha=1/(1+\beta_k)$, and with a positive real number $\Lambda>1$ (the $(1+\epsilon)$ of that theorem) to be determined subsequently,  we obtain:
\begin{equation}\label{B}
\sup_M\varphi
\leq\beta_k\fint_M (-\varphi) \ \omega_{\varphi}{}^n
+\Lambda I_k(\varphi)+ C.
\end{equation}

Consider the following nonnegative linear combination of the inequalities (\ref{C}), (\ref{A}), (\ref{B}):
\begin{equation}\label{lincomb}
\beta_1\Lambda(\ref{C})
+\left[\Lambda-(n-k+1)\beta_k\right](\ref{A})
+(n-k+1)\beta_1(\ref{B}).
\end{equation}
(Since $\Lambda>1>\frac{n-k+1}{n}>(n-k+1)\beta_k$, and $\beta_1\geq 0$, the coefficients of (\ref{C}), (\ref{A}), (\ref{B}) are indeed nonnegative.)

After moving everything to the left-hand side, the coefficient of $I_k(\varphi)$ in inequality (\ref{lincomb}) is:
\[
\beta_1\Lambda\cdot (n-k+1)
+ (n-k+1)\beta_1\cdot -\Lambda
=0.
\]
The coefficient of $\fint_M (-\varphi) \ \omega_{\varphi}{}^n$ in inequality (\ref{lincomb}) is:
\[
\beta_1\Lambda\cdot 1
+\left[\Lambda-(n-k+1)\beta_k\right]\cdot -\beta_1
+(n-k+1)\beta_1\cdot -\beta_k=0.
\]
The coefficient of $\sup_M\varphi$ in inequality (\ref{lincomb}) is:
\begin{eqnarray*}
&&\beta_1\Lambda\cdot -n
+\left[\Lambda-(n-k+1)\beta_k\right]\cdot 1
+(n-k+1)\beta_1\cdot 1\\
&=&(n\Lambda-n+k-1)\left(\tfrac{1}{n}-\beta_1\right)
+(n-k+1)\left(\tfrac{1}{n}-\beta_k\right).
\end{eqnarray*}

We now make a specific choice of $\beta_1$, $\beta_k$ and $\Lambda$.  Rearranging the alpha-invariants criterion gives
\[
 (k-1)\left[\frac{1-\alpha_{m,1}(M)}{\alpha_{m,1}(M)}-\frac{1}{n}\right]
 <(n-k+1)\left[\frac{1}{n}-\frac{1-\alpha_{m,k}(M)}{\alpha_{m,k}(M)}\right].
\]
Since this holds, we may choose $\beta_1$, $\beta_k$ and $\Lambda=1+\epsilon$ satisfying the conditions stated previously, and such that moreover
\[
 (n[\Lambda-1]+k-1)\left[\beta_1-\frac{1}{n}\right]
 <(n-k+1)\left[\frac{1}{n}-\beta_k\right].
\]

The coefficient of $\sup_M\varphi$  in the inequality (\ref{lincomb}) is then positive, and the inequality simplifies to the uniform bound
\[
\sup_M\varphi\leq
\left\{(n-k+1)\left[\frac{1}{n}-\beta_k\right]
-(n[\Lambda-1]+k-1)\left[\beta_1-\frac{1}{n}\right]
 \right\}^{-1}C.
\]

\end{proof}

\subsection{Existence of K\"ahler-Einstein metrics}

\begin{corollary} \label{general-existence}
Suppose that
\begin{enumerate}
\item (partial $\mathcal{C}^0$-estimate) there exists $a>0$,
such that for all $t\in \mathcal{I}\cap [\delta,1)$, $ \rho_{\omega_{\varphi_t},m}\geq a$;
\item (alpha-invariants criterion)
\begin{eqnarray*}
 \frac{n}{n+1}&<&\alpha_{m,k}(M),\\
 (k-1)\left[\frac{1}{\alpha_{m,1}(M)}-\frac{n+1}{n}\right]
 &<&(n-k+1)\left[\frac{n+1}{n}-\frac{1}{\alpha_{m,k}(M)}\right];
\end{eqnarray*}
\item the $k$-th eigenvalue of $K_M^{-m}$ is controlled.
\end{enumerate}
Then $M$ admits a K\"ahler-Einstein metric.
\end{corollary}

\begin{proof}
Choose $\delta$ sufficiently small that $[0,\delta]\subseteq \mathcal{I}$ is nonempty.  By arguments due to Aubin \cite{Aub76} and Yau \cite{yau77, yau78},  the a priori $\mathcal{C}^0$-bound of Theorem \ref{c0-estimate} implies an a priori $\mathcal{C}^{2,\gamma}$ bound.  Hence the set $(\varphi_t)_{t\in\mathcal{I}\cap[\delta,1)}$ is precompact in the $\mathcal{C}^2$ topology, so $\mathcal{I}$ is closed in $[0,1]$.    We already knew $\mathcal{I}$ was open in $[0,1]$ and contained 0; thus it contains 1.
\end{proof}

Combining Corollary \ref{general-existence} with Szekelyhidi's partial $\mathcal{C}^0$-estimate Theorem \ref{partial-c0} yields Theorem \ref{main-thm}.

\begin{remark}
If the $k$-th eigenvalue of $K_M^{-m}$ is instead only weakly controlled, then the same arguments show that there exists an $\epsilon>0$, not particularly computable or nice, such that if
\begin{eqnarray*}
 \alpha_{m,k}(M)&>&\frac{n}{n+1},\\
 \alpha_{m,1}(M)&>&\frac{n}{n+1}-\epsilon,
\end{eqnarray*}
then $M$ admits a K\"ahler-Einstein metric.  In particular if
\begin{eqnarray*}
 \alpha_{m,k}(M)&>&\frac{n}{n+1},\\
 \alpha_{m,1}(M)&=&\frac{n}{n+1},
\end{eqnarray*}
then $M$ admits a K\"ahler-Einstein metric.

Specifically, if the $k$-th eigenvalue of $K_M^{-m}$ is weakly controlled, so that we have a uniform inequality of the form
\[
\log\left[\frac{\mu_1(a)}{\mu_{k}(a)}\right]
\leq \Lambda I(\psi_a) + C
\]
(for some $\Lambda>1$), then K\"ahler-Einstein metrics exist when the following stronger, $\Lambda$-dependent alpha-invariants criterion is satisfied:
\begin{eqnarray*}
 \frac{n}{n+1}&<&\alpha_{m,k}(M),\\
 (n[\Lambda-1]+k-1)\left[\frac{1}{\alpha_{m,1}(M)}-\frac{n+1}{n}\right]
 &<&(n-k+1)\left[\frac{n+1}{n}-\frac{1}{\alpha_{m,k}(M)}\right].
\end{eqnarray*}
\end{remark}


\bibliographystyle{alpha}
{\small\bibliography{kaehler}{}}

\end{document}